\documentclass[11pt,letterpaper]{article}

\usepackage{amsmath,tikz}
\usepackage{amsfonts}
\usepackage{amssymb}
\usepackage{amsthm}
\usepackage{enumerate}
\usepackage{fullpage}
\usepackage{xcolor}
\definecolor{ForestGreen}{rgb}{0.1333,0.5451,0.1333}
\definecolor{Red}{rgb}{0.9,0,0}
\usepackage[linktocpage=true,
	colorlinks,
	linkcolor=blue,citecolor=ForestGreen,
	bookmarks,bookmarksopen,bookmarksnumbered]
	{hyperref}
\usepackage{cleveref}
\usepackage{comment}
\usepackage{ifthen}

\makeatother

\begin{document}

\newtheorem*{thm}{Theorem I}

\newtheorem*{theorem*}{Theorem}
\newtheorem{theorem}{Theorem}
\newtheorem{lemma}{Lemma}

\newtheorem{definition}{Definition}
\newtheorem{corollary}{Corollary}
\newtheorem{example}{Example}

\newtheorem{remark}{Remark}

\newcommand{\summ}{\displaystyle\sum}
\newcommand{\prodd}{\displaystyle\prod}
\newcommand{\F}{\mathbb{F}}
\newcommand{\R}{\mathbb{R}}
\newcommand{\E}{\mathbf{E}}
\newcommand{\Var}{\mathbf{Var}}
\newcommand{\mon}{\mathcal{M}}
\newcommand{\poly}{\mathrm{poly}}
\newcommand{\PA}{\widetilde{A}}
\newcommand{\PB}{\widetilde{B}}
\newcommand{\PC}{\widetilde{C}}
\newcommand{\Palpha}{\widetilde{\alpha}}
\newcommand{\Pbeta}{\widetilde{\beta}}
\newcommand{\Pgamma}{\widetilde{\gamma}}

\newcommand{\zo}[1]{\ensuremath{\{0,1\}^{#1}}}

\newcommand{\etao}{\ensuremath{\eta_0}}
\newcommand{\somesrc}{\ensuremath{\textnormal{somewhere random source}}}
\newcommand{\kextractor}{\ensuremath{(k,\epsilon)\textnormal{-extractor}}}
\newcommand{\swsrc}{\ensuremath{s\textnormal{-where random source}}}
\newcommand{\swsrcs}{\ensuremath{s\textnormal{-where random sources}}}
\newcommand{\exmerger}{\ensuremath{\epsilon\textnormal{-extracting merger}}}
\newcommand{\exmultimerger}{\ensuremath{(\epsilon,s)\textnormal{-extracting multimerger}}}
\newcommand{\exmultimergers}{\ensuremath{(\epsilon,s)\textnormal{-extracting multimergers}}}
\newcommand{\kedisp}{\ensuremath{(k,\epsilon)\textnormal{-disperser}}}

\title{Extracting Mergers and Projections of Partitions}
\author{\hfill Swastik Kopparty\thanks{Department of Mathematics and Department of Computer Science, University of Toronto.\newline \textcolor{white}{--------}Research supported by an NSERC Discovery Grant. \newline \textcolor{white}{--------}Email: \href{mailto:swastik.kopparty@utoronto.ca}{swastik.kopparty@utoronto.ca}  \newline \textcolor{white}{----...}} \and Vishvajeet N\thanks{School of Informatics, University of Edinburgh.  \newline   \textcolor{white}{---.----}Research supported by the European Research Council (ERC) under the European Union's  \newline  \textcolor{white}{----...} Horizon 2020 research and innovation programme (grant agreement No. 947778).  \newline \textcolor{white}{----...} Email: \href{mailto:nvishvajeet@gmail.com}{nvishvajeet@gmail.com}}}
\date{}
\maketitle

\begin{abstract}
We study the problem of extracting randomness from somewhere-random sources, and related combinatorial phenomena: partition analogues of Shearer's lemma on projections.

A somewhere-random source is a tuple $(X_1, \ldots, X_t)$ of (possibly correlated) $\{0,1\}^n$-valued
random variables $X_i$ where
for some unknown $i \in [t]$, $X_i$ is guaranteed to be uniformly distributed. 
An {\em extracting merger} is a seeded device that takes a somewhere-random source as input and
outputs nearly uniform random bits. We study the seed-length needed for extracting mergers
with constant $t$ and constant error.

Since a somewhere-random source has min-entropy at least $n$, a standard extractor can also serve
as an extracting merger. Our goal is to understand whether the further structure of being somewhere-random
rather than just having high entropy enables smaller seed-length, and towards this we show:
\begin{itemize}
\item Just like in the case of standard extractors, seedless extracting mergers with even just one output bit do not exist.
\item Unlike the case of standard extractors, it {\em is} possible to have extracting mergers that output a constant number of bits using only constant seed. Furthermore, a random choice of merger does not work for this purpose!
\item Nevertheless, just like in the case of standard extractors, an extracting merger which gets most of the entropy out (namely, having $\Omega(n)$ output bits) must have $\Omega(\log n)$ seed. This is the main technical result of our work, and is proved by a second-moment strengthening of the graph-theoretic approach of Radhakrishnan and Ta-Shma to extractors.
\end{itemize}

All this is in contrast to the status for condensing mergers (where the output is only required to have
high min-entropy), whose seed-length/output-length tradeoffs can all be fully explained by using standard condensers.

Inspired by such considerations, we also formulate a new and basic class of  problems in combinatorics: partition analogues of Shearer's lemma.
We show basic results in this direction; in particular, we prove that in any partition
of the $3$-dimensional cube $[0,1]^3$ into two parts, one of the parts has an axis parallel
 $2$-dimensional projection of area at least $3/4$.

\end{abstract}

\clearpage
\tableofcontents
\newpage 
\section{Introduction}

We study the problem of extracting randomness from somewhere-random sources, and related combinatorial phenomena: partition analogues of Shearer's lemma on projections. For the (completely self-contained) combinatorics, see~\Cref{combintro}, \Cref{sec54} and \Cref{sec6}.

A $t$-part somewhere-random source is a tuple $(X_1, \ldots, X_t)$ of (possibly correlated) $\zo{n}$-valued random variables $X_i$, where some unknown $X_i$ is guaranteed to be uniformly distributed. 
We will take $t$ to be constant and $n$ growing throughout this paper.
A {\em merger} is a seeded device that takes a somewhere-random source and purifies its randomness.
Mergers have been extensively studied in the theory of extractors, and have played an important
role in their development. In fact, there were at least 3 distinct points in the history of extractors~\cite{TaShma,LRVW,DKSS}
when the best known explicit extractor constructions were based on new advances in explicit merger constructions.

An important observation is that $t$-part somewhere-random sources are special cases of sources with (min) entropy rate $1/t$.
Thus any randomness purifying device (such as an extractor, condenser or disperser) that can give guarantees
when fed a source with entropy rate at least $1/t$ is automatically some kind of merger for $t$-part somewhere-random sources.

In the literature, mergers have only been studied in the {\em condensing} regime: where their output
is required to have high entropy rate (rather than requiring the output to be near-uniform).
It turns out that information-theoretically, condensing mergers are completely overshadowed
by classical condensers. A condenser is a seeded device that takes in a source with sufficient
entropy rate and outputs a random variable with high entropy rate. 
Thus a condenser that can operate on sources with entropy rate $1/t$ is automatically a 
condensing merger for $t$-part somewhere-random sources.
It turns out that
 whatever parameter ranges are achievable by condensing mergers can be completely explained by condensers.

In this paper, we study mergers in the {\em extracting} regime: 
where their output is required to be near-uniform. Our main result is a characterization of the seed-length
needed for such extracting mergers. Unlike the tragic case of condensing mergers and their
relationship with condensers, extracting mergers are able to step out of the shadow of extractors, and carve
a niche, albeit small, for themselves.

We also study extracting multimergers, where more random variables out of the given tuple of random variables are required to be uniform and independent. This leads us to a number of interesting combinatorial / geometric questions, for which we give some new and basic combinatorial theorems (such as a partition analogue of Shearer's lemma on projections of a set in a product space).

\subsection{Overview of results}

Our results are best viewed in contrast to the situation with classical extractors and condensers.
An extractor takes a source with some min-entropy and an independent uniform seed, and outputs a nearly-uniform distributed random variable.
A condenser takes a source with some min-entropy and an independent uniform seed, and outputs a source with higher min-entropy-rate.

Both extractors and condensers are functions of the form:
$$F : \zo{n} \times \zo{d} \to \zo{m},$$
where $d$ is the ``seed-length" and $m$ is the ``output-length".

Consider a random source $\mathbf X$ that is $\zo{n}$-valued
and has entropy rate $1/t$ (which means that its min-entropy is $\geq n/t$).

In the case of extractors, for $(1-\epsilon)$-fraction of $j \in \zo{d}$, the output $F(\mathbf{X}, j)$
is required to be $\epsilon$-close in statistical distance to the uniform distribution over $\zo{m}$.
In the case of condensers, for $(1-\epsilon)$-fraction of $j \in \zo{d}$, the output $F(\mathbf{X},j)$
is required to be $\epsilon$-close in statistical distance to some $\zo{m}$-valued random variable
with min-entropy $\geq k'$.

Extractors and condensers are qualitatively very different from the point of view of seed-length. We summarize their salient features below:
\begin{itemize}
\item There are no seedless extractors or condensers.
\item There are condensers with {\bf constant} seed-length $d = O(\log \frac{1}{\epsilon})$ which are {\bf lossless} (we can take $k'$ as large as $\frac{n}{t}+d$), provided $m > k' + \Omega(\log \frac{1}{\epsilon})$.
\item The seed-length required for an extractor
to extract one bit of entropy from a random source $(\zo{n})^t$ is $\log n + 2 \log \frac{1}{\epsilon} + O(1)$.
Furthermore, this seed-length suffices to extract almost all the entropy out of the source.
\end{itemize}

A merger takes in a $t$-part somewhere-random source (which is a special case of a source with entropy rate $\frac{1}{t}$) and an independent uniform seed, and outputs a source with purer randomness. 
This naturally creates two kinds of mergers - condensing mergers and extracting mergers. To the best of our knowledge, only condensing mergers have been studied in the literature, and the (non-constructive) existence results for condensing mergers all follow from the existence results for condensers mentioned above.

Let  $E : (\zo{n})^t \times \zo{d} \to \zo{m}$ be an extracting merger,
namely its output is guaranteed to be $\epsilon$-close to uniform on $\zo{m}$ whenever given
a $t$-part somewhere-random source as input.

\noindent
{\bf Theorem A (Informal):}
{\em We have the following:
\begin{itemize}
\item There are no seedless extracting mergers, even with output length $1$.
\item There are extracting mergers with {\bf constant} seed length $O(\log \frac{1}{\epsilon})$, which can output
a constant number of nearly-uniform bits. 
\item Nevertheless, if the seed length required for an extracting merger
to extract almost all (or even a constant fraction) the entropy out of a somewhere-random source is $\Theta(\log n)$.
\end{itemize}
}

The first item is trivial. The second item is also not difficult, but it already gives a taste of why things are different with extracting mergers. Indeed, randomly-chosen functions are not extracting mergers.
The third item in the above theorem is our main technical result.
It is proved by a second-moment strengthening of the graph-theoretic approach of Radhakrishnan and Ta-Shma to extractors.

\subsection{Projections of partitions}
\label{combintro}

Our study of these questions about randomness extraction leads us to formulate and make progress on a new and natural combinatorial question: the partition analogue of the Shearer/Loomis-Whitney inequalities on volumes of  projections. These questions arise when we consider the problem of extracting randomness from
$t$-part $s$-where random sources (where $s$ out of the $t$ parts of the source are uniform and independent). We call devices that
do this {\em extracting multimergers}. For the rest of this subsection we only focus on the combinatorial aspect.

Let $A$ be an ``nice" subset of the solid cube $[0,1]^3$ with (Lebesgue) volume $\alpha$. Consider the three axis-parallel 2-dimensional projections: 
$\Pi_{XY}(A)$, $\Pi_{YZ}(A)$, $\Pi_{XZ}(A)$.
The Shearer/Loomis-Whitney inequality~\cite{CGFS86, LoomisWhitney} implies that at least one of these three projections has area at least $\alpha^{2/3}$. This is tight, as witnessed by the case where  $A$ is a cube of side-length $\alpha^{1/3}$ (and this is roughly the only such example).

Now consider the following partition variant: Let $A,B$ be ``nice" subsets of $[0,1]^3$ that partition $[0,1]^3$. Consider the six axis-parallel 2-dimensional projections of these two sets:
$\Pi_{XY}(A)$, $\Pi_{YZ}(A)$, $\Pi_{XZ}(A)$ and
$\Pi_{XY}(B)$, $\Pi_{YZ}(B)$, $\Pi_{XZ}(B)$. 
How large can we guarantee that one of them is?

Using the previous inequality and the fact that at least one of $A,B$ has volume at least $1/2$, we get that one of these six 2-dimensional projections has area at least $(1/2)^{2/3} \geq 0.6299$. For this bound to be tight, we would need both $A$ and $B$ to have volume $1/2$, and both $A$ and $B$ to be tight examples for the Shearer/Loomis-Whitney inequality. This would require us to be able to cover $[0,1]^3$ by two cubes of volume $1/2$ -- which is clearly impossible. This suggests that there should be a better bound!

We show, using a delicate study of the sections of the cube and some seemingly lucky inequalities, a tight bound for this problem.

\noindent
{\bf Theorem B (Informal): }\quad Let $A, B$ be ``nice" subsets of $[0,1]^3$ that partition $[0,1]^3$.
Then at least one of the six 2-dimensional projections
$$\Pi_{XY}(A), \Pi_{YZ}(A), \Pi_{XZ}(A), \Pi_{XY}(B), \Pi_{YZ}(B), \Pi_{XZ}(B),$$
has area at least $3/4$.

Such ``projections of partitions'' questions can be formulated in great generality, and apart from Theorem B (whose proof we find very interesting),
we also make some general observations and make some slightly non-trivial progress. We think these are very natural combinatorial questions worthy of
further study. Beyond having connections to mergers, these questions turn out to be related to the KKL and BKKKL theorems/conjectures~\cite{KKL,BKKKL,Friedgut,FHHHZ}
on influences of Boolean functions on the solid cube $[0,1]^n$.
For example, Theorem B implies that any 3-variable Boolean function $f:[0,1]^3 \to \{0,1\}$ has some variable and some bit $b$ such that the ``influence towards $b$" of that variable is at least $1/4$, and this is tight.

Another application of such results is to partition analogues of the Kruskal-Katona theorem. For example, Theorem B implies that for any partition of ${ [n] \choose 3 }$ into two parts, one of the two parts has shadow with size at least $\left(\frac{3}{4} - o(1)\right) {n \choose 2}$.

\subsection{Related work}

Mergers were introduced by Ta-Shma~\cite{TaShma} in his thesis, 
and were used to construct state-of-the-art extractors at the time (these were condensing mergers).
Later, \cite{LRVW} proposed a new condensing merger construction based on 
taking random linear combinations of vectors over finite fields,
and used it in their construction of the first extractors optimal upto
constant factors. This analysis was greatly improved by Dvir~\cite{Dvir09}
through his solution to the finite field Kakeya conjecture. 
Subsequently, \cite{DvirWigderson, DKSS} defined a higher degree polynomial
variant of the \cite{LRVW} merger, and by developing the ideas from~\cite{Dvir09},
were able to construct improved (constant seed) mergers and state-of-the-art extractors. Subsequently \cite{TU} showed how to get analogous explicit constructions of condensers (subsuming the \cite{DKSS} condensing mergers) 
by improving the \cite{GUV} condensers.

Another interesting constant seed condensing merger is by \cite{Raz05}, which was constructed on the way to multi-source extractors.

Our lower bounds for the seed length of extracting mergers are proved by developing ideas from the paper of Radhakrishnan and Ta-Shma~\cite{RTaShma}. A recent beautiful proof of \cite{AGORSO} also achieved a similar result to \cite{RTaShma} in a much cleaner way, but we were not able to adapt this approach to our setting.

Other papers relevant to the study of multimergers are related to resilient functions~\cite{CGGL, CZ19, Meka}.

Finally, our combinatorial results are related to the  KKL and BKKKL theorems/conjectures~\cite{KKL,BKKKL,Friedgut,FHHHZ}
on influences of Boolean functions on the solid cube $[0,1]^n$.

\subsection{Organization}

We give the basic definitions of extracting mergers and extracting multimergers in \Cref{sec2}. In \Cref{sec3} we start with a simple proof that seedless mergers do not exist. This is followed by showing the existence of mergers and multimergers in the extracting regime with constant seed-length.
We prove our lower bound on the seed length of extracting mergers in \Cref{sec4}, which culminates in \Cref{thm:mergerlowerbound}. 
In \Cref{sec5} we explore the connection between {\em seedless} extracting mergers and projections of partition questions. \Cref{sec54} is devoted to proving \Cref{lem-3-2-2parts}, our (optimal) lower bound on partitioning the unit cube into $2$ parts, and \Cref{sec6} is devoted to partitions of the cube into $3$ parts.

\pagebreak 
\section{Sources and Mergers}\label{sec2}

\begin{definition} [$k$-source]
For any $k$, we say that a random variable $X$ is a $k$-source if for all $x$,
 $\Pr[X=x]\leq 2^{-k}$
\end{definition}

\subsection{Somewhere and \texorpdfstring{$s$}{s}-where Random Sources}

\begin{definition}[Somewhere-Random Source]\label{def:somewhere-src}
For a domain $D$, a tuple ${\mathbf X} = (X_1, \ldots, X_t)$ of jointly
distributed $D$-valued
random variables is called a $t$-part somewhere random source if for some
$i \in [t]$, the distribution
 of $X_i$ is uniform over $D$.
 \end{definition}

\begin{definition}[$s$-where Random Source]\label{def:s-wheresrc}
For a domain $D$ and an integer $s > 0$, a tuple ${\mathbf X} = (X_1, \ldots, X_t)$ of jointly distributed $D$-valued
random variables is called a $t$-part $s$-where random source if for some distinct $i_1, \ldots, i_s \in [t]$, the
joint distribution
 of $(X_{i_1}, X_{i_2}, \ldots, X_{i_s})$ is uniform over $D^s$. 
\end{definition}

\subsection{Extracting Mergers and Multimergers}

\begin{definition}[Extracting Mergers]\label{def-ntd-merger}
Let $n, t, d, m$ be integers, and let $\epsilon > 0$.

A function $E: \left(\zo{n}\right)^t \times \zo{d} \to \zo{m}$ is called an 
\underline{$(n,t,d,m,\epsilon)$-extracting merger} if the following holds.

Suppose ${\mathbf X} = (X_1, \ldots, X_t)$ is a somewhere-random source where each $X_i$ is $\zo{n}$-valued.
Then for at least $(1-\epsilon)$-fraction of $j \in \zo{d}$, the distribution of:
 $$ Z = E( {\mathbf X}, j),$$
 is $\epsilon$-close to the uniform distribution on $\zo{m}$.
\end{definition}

We will sometimes refer to these as $\epsilon$-extracting mergers (since $n,d,t,m$ are related to the shape of 
$E$).

\begin{definition}[Extracting Multimergers]\label{def-ex-multimerger}
Let $n,t,s,d,m$ be integers, and let $\epsilon > 0$.

A function $E: \left(\zo{n}\right)^t \times \zo{d} \to \zo{m}$ is called an 
\underline{$(n,d,t,m,\epsilon,s)$-extracting multimerger} if the following holds.

Suppose ${\mathbf X} = (X_1, \ldots, X_t)$ is an $s$-where random source 
where each $X_i$ is $\zo{n}$-valued.
Then for at least $(1-\epsilon)$-fraction of $j \in \zo{d}$, the distribution of:
 $$ Z = E( {\mathbf X}, j),$$
 is $\epsilon$-close to the uniform distribution on $\zo{m}$.
\end{definition}

We will sometimes refer to these as $(\epsilon,s)$-extracting multimergers (since $n,d,t,m$ are related to the shape of 
$E$).

Observe that the $s = 1$ case in the above definition corresponds to extracting mergers.  

\paragraph{Note on the definitions} In all our definitions, we chose to define the ``strong" versions (where the output bits are required to be independent of the seed) for simplicity. In fact, our existence result for mergers is for the strong version, and our impossibility result is for the weak version.


\section{Simple results about extracting mergers}\label{sec3}

For the rest of this paper, we only talk about {\em extracting} (not condensing) mergers and multimergers.

\subsection{Seedless Mergers do not exist}\label{sec31}

We begin with the simple observation that there are no seedless 
extracting mergers.
\begin{theorem}(There are no seedless mergers)
Let $n$ be an integer and $\varepsilon < 1/2$. There does not exist a function $M: \zo{n}\times \zo{n} \to \zo{}$ that is an $\varepsilon$-merger.
\end{theorem}
\begin{proof}
	Fix an $\varepsilon < 1/2$. Assume for the sake of contradiction there exists an $\varepsilon$-merger $M: \zo{n} \times \zo{n}  \to  \zo {} $. 
	
	In particular, this means for {\em every} function $f: \zo{n} \to \zo{n}$, when  $X$ is distributed uniformly over $\zo{n}$, the distribution of $ M(X, f(X))$ is $\varepsilon$-close to uniform
on $\zo{}$ -- and in particular, it has full support on $\zo{}$. We will now demonstrate a function $g:\zo{n} \to \zo{n}$ such that $M(g(Y), Y)$ is constant for uniformly distributed $Y$, thus contradicting the merger assumption.

Fix any $y\in \zo{n}$. Consider the constant function $f_y: \zo{n} \to \zo{n}$ given by $f_y(x) = y$ for all $x$. By our hypothesis above, the distribution of $M(X,f_y(X))$ has full support $\zo{}$.
Thus there exists $x \in \zo{n}$ such that $M(x,y) = 0$. Pick one such $x$ and call it $g(x)$.

Thus we have $M(g(y), y) = 0$ for all $y \in \zo{n}$. We conclude that for uniform $Y \in \zo{n}$, $M(g(Y), Y) = 0$, which is the desired contradiction.
\end{proof}

\subsection{Extracting mergers with constant seed exist}\label{sec32}

We now show that constant seed extracting mergers with constant output length exist. While the proof is quite simple, it is interesting because (1) constant seed extractors do not exist, (2) a random choice of $E:(\zo{n})^t\times\zo{d} \to \zo{m}$ does not give a constant seed extracting mergers, and most importantly (3) as we will later see, the seed length still needs to be superconstant to produce a superconstant number of output bits, as we will see in the next section.

\begin{theorem}\label{thm:mergerupperbound}
Let $n,t$ be integers and $\epsilon > 0$.

Then for any integer $m \leq n $, setting:

$$d = \log m + \log (t-1) + 2 \log \frac{1}{\epsilon} + O(1),$$

there exists a function $E: \left(\zo{n}\right)^t \times \zo{d} \to \zo{m}$
that is an $\varepsilon$-extracting merger.
\end{theorem}

Thus with $O(\log t + \log\frac{1}{\epsilon})$ bits of seed, we can extract $\poly(\frac{1}{\epsilon})$ bits out.

\begin{proof}
We want to get an extracting merger $E((x_1, \ldots, x_t), j)$, where
the $x_i \in \zo{n}$ and $j \in \zo{d}$.

 The nature of a somewhere-random source is that applying a
 truncation to
 each element of the source yields a smaller somewhere-random source.
 The idea of our extracting merger is to truncate
 our somewhere-random source, and to then apply a standard
 seeded extractor to
 the entire truncated source. The truncation makes the instance
 size smaller, enabling us to use a reduced seed length in the extractor.

 We truncate each $x_i$ to the first $m$ bits, thus obtaining
 $x_1', \ldots, x_t' \in \zo{m}$.

We can verify that our truncation to the first $m$ bits produces a source $(X_1', \ldots, X_t')$ of length $mt$ and min-entropy $m$.  By the standard result on existence of extractors (See Theorem 6.14 in~\cite{Vadhan}), there
 exists a strong $(m, \epsilon)$-extractor
 $Ext_0 : \zo{m t} \times \zo{d} \to \zo{m}$ with seed length $d= \log m + \log(t-1) + 2 \log \frac{1}{\epsilon} + O(1)$.

We can thus define the function $E: (\zo{n})^t \times \zo{d} \to \zo{m}$:
$$E((x_1, \ldots, x_t), j) =  Ext_0((x_1', \ldots, x_t'), j).$$
Observe that the function $E$ is
 an $\epsilon$-extracting merger that
 uses a seed $j$ of length $d$ and outputs $m$ bits as required.
\end{proof}

In contrast, a random $E: (\zo{n})^t \times \zo{d} \to \zo{}$ is not an extracting merger at all!
To see this, it suffices to fix $t=2$. 
If $E$ is chosen at random, then for every $j \in \zo{d}$ and $x \in \zo{n}$, it is very likely that there
exists a $y \in \zo{n}$ such that $E((x,y),j) = 0$. Define $f_j:\zo{n} \to \zo{n}$ by $f_j(x)$ = any such $y$.
Then for every $j \in \zo{d}$, $E(X, f_j(X), j)$ is constant when $X$ is picked uniformly at random, showing that
$E$ is not a merger.

\subsection{Extracting Multimergers}\label{sec33}

Using the same idea, we also get interesting multimergers.

\begin{theorem}
\label{thm:multimergerupperbound}
Let $n,t,s$ be integers with $s < t$, and $\epsilon > 0$.
Then for any integer $a \leq n$, setting $m = s \cdot a$ and:
$$ d = \log a + \ 2 \log \frac{1}{\epsilon} + \log(t-s) + \Omega(1),$$
there exists a function $E: \left(\zo{n}\right)^t \times \zo{d} \to \zo{m}$
that is an $(\varepsilon,s)$-extracting multimerger.
\end{theorem}

Taking for example $s = t-1$ and $a = \poly\left(\frac{1}{\epsilon}\right) \ll n$, 
we get that by investing $O(\log \frac{1}{\epsilon})$ bits of seed, we can extract 
$\poly\left(\frac{1}{\epsilon}\right) \cdot t$ bits of randomness from any
$t$-part $(t-1)$-where random source $\mathbf X \in \left(\zo{n}\right)^t$.

In this setting of parameters, the seed length does not even depend on $t$, and we could take $t$
to be growing superconstantly while preserving constant seed-length.

\subsection{Seedless Multimergers}\label{sec34}

Our final observation of this section is that for multimergers with large $t$ and where $s$ is a large fraction of $t$, seedless
multimergers with small error {\em do} exist. Indeed, if $s = t-1$, and we define
$E: (\zo{n})^t \to \zo{}$ by
$$E(x_1, \ldots, x_t) = Maj(x_{11}, x_{21}, \ldots, x_{t1}),$$
it is easy to see that $E$ is a seedless $(\epsilon, t-1)$-multimerger for $\epsilon = O(\frac{1}{\sqrt{t}})$.
Replacing $E$ with any resilient function gives other examples of seedless multimergers (including with larger output size).

Investigation of this phenomenon leads us to the projections of partitions question, and we explicitly give the connection and some
results about it in a later section. Nevertheless, this seems like the tip of an iceberg.

\section{Mergers with large output need large seed}\label{sec4}

In this section we show a lower bound on the seed-length for $2$-source
extracting mergers, essentially showing that the dependence of the 
seed-length for extracting mergers in Theorem~\ref{thm:mergerupperbound}
on $m$ and $\epsilon$ is tight. 

\begin{theorem}
\label{thm:mergerlowerbound}
Let $\epsilon < 1/40$.
Let $E: \left(\zo{n}\right)^2 \times \zo{d} \to \zo{m}$ be a
$\varepsilon$-extracting merger. 

Then for $\epsilon \geq 2^{-\Omega(m)}$, we have: 
$$d \geq \log m + \log \frac{1}{\epsilon} - O(1).$$
and for $\epsilon < 2^{-\Omega(m)}$, we have:
$$d \geq \Omega(m).$$
\end{theorem}

For the proof of this theorem, the representation of the inputs and output
of $E$ in terms of bits
is a distraction. So, letting $N = 2^n$, $D=2^d$, $M= 2^m$ and identifying
$\zo{n}, \zo{d}, \zo{m}$ with $[N], [D], [M]$ respectively, we will
view $E$ as a function $E : [N]^2 \times [D] \to [M]$.

Recalling the $\epsilon$-extracting merger property, we have that $E$
is such that whenever $X,Y$ are jointly distributed $[N]$-valued random
variables, with at least one of them uniformly distributed, 
and $J$ is picked uniformly from $[D]$ and independently of $(X,Y)$,
then the distribution of $E((X,Y), J)$ is $\epsilon$-close to the uniform
distribution on $[M]$.

We will show that for $\epsilon \geq M^{-\Omega(1)}$, we have: 
$$D \geq \Omega\left(\frac{1}{\epsilon}\log M\right),$$
and for $\epsilon < M^{-\Omega(1)}$, we have:
$$D \geq \Omega\left(M^{\Omega(1)}\right) .$$
 
 Our proof is based on the following idea.
Consider a uniformly random subset $S \subseteq [M]$ of size $\lambda M$.
For each $y \in [N]$, we look for an $x$ such that for all
$j \in [D]$, $E(x,y,j) \not\in S$. If there is such an $x$, 
then we define $g(y) = x$. If such an $x$ exists for most $y$, then 
for uniformly chosen $Y \in [N], J \in [D]$,
we have $\Pr_{Y,J}[E(g(Y),Y,J) \in S] \ll \lambda - \epsilon$, contradicting the merger property. Thus for most $S$, for many $y$ there is no such $x$; namely, for most $S$, for many $y$, for all $x$, there is some $j$, such that
$E(x,y,j) \in S$. For this to happen for even one $y$ turns out to be very abnormal, and we derive our lower bound on $D$ by digging into its structure. This part uses a second moment variation of the Radhakrishnan-TaShma~\cite{RTaShma} approach to extractor lower bounds.

 \subsection{Abnormal conductors}\label{sec41}

A map $C : [N] \times [D] \to [M]$ is called a conductor (this is a general term capturing the shape of seeded extractors and seeded condensers). We will also view this as a bipartite multigraph with $[N]$ on the left, $[M]$ on the right and $D$ labelled edges coming out of every left vertex.
 
 If $C$ is chosen at random, then for most $S \subseteq [M]$ of size $\lambda M$ and most $x \in [N]$, we expect about $\lambda$ fraction of the edges coming out of $x$ to land in $S$. But we do not expect that this will happen for {\em all} $x$! When $C$ is chosen at random, then for most $S$ there will be some $x \in [N]$ for which a very small $( \ll \lambda)$ fraction of edges coming out of $x$ lie in $S$. We capture this with the following definition.

 \begin{definition} 
 Let $C: [N] \times [D] \to [M]$ be a conductor.
 Let $S$ be a subset of $[M]$.
 
 We say the vertex $x \in [N]$ {\em totally misses} $S$ (under $C$) if 
 $$\left|\{j \in [D] \mid C(x,j) \in  S\}\right| = 0.$$ 
 
 We say the vertex $x \in [N]$ {\em mostly misses} $S$ (under $C$) if 
 $$\left|\{j \in [D] \mid C(x,j) \in S\}\right| <  \frac{1}{2} \frac{|S|}{M} D.$$ 
 \end{definition}
 
 \begin{definition}[Abnormal conductors]
 Let $C: [N] \times [D] \to [M]$ be a conductor.
 
We say that $C$ is $(\gamma, \lambda)$-abnormal if 
 $$ \Pr_{S \in {[M] \choose \lambda M } } \left[ \exists x \in [N] \mbox{ s.t. } x \mbox{ mostly misses } S \right] < 1- \gamma.$$
 \end{definition}

 \begin{lemma}[Extracting mergers contain abnormal conductors]\label{lem:findabnormal}
  Suppose $0 < \gamma < \frac{\lambda}{2} - \epsilon$. 
  Suppose $E: [N]^2 \times [D] \to [M]$ is an $\epsilon$-extracting merger.
  For $y \in [N]$, let $E_y: [N] \times [D] \to [M]$ be the function $E( \cdot, y, \cdot)$. Then for some $y \in [N]$,
  $E_y$ is $(\gamma, \lambda)$-abnormal.
 \end{lemma}
 \begin{proof}
  Suppose not; namely that for all $y \in [N]$, we have that $E_y$ is not $(\gamma, \lambda)$-abnormal.
  
  Pick $S \in {[M] \choose \lambda M}$ uniformly at random.
  
  Let $B_y$ be the event that there exists some $x \in [N]$ that mostly misses $S$ under $E_y$.
  
  By our assumption, $\Pr[B_y] \geq 1-\gamma$.
  So the expected number of $y$ for which $B_y$ happens is at least $(1-\gamma)N$.
  
  Thus there exists some particular choice of $S$ for which $B_y$ happens for
  at least $(1-\gamma)N$ many $y$s.
  Call this choice  $S_0$. Define $f: [N] \to [N]$ by defining $f(y)$ as follows:
  $$ f(y) =
  \begin{cases} \mbox{ any $x$ that mostly misses $S_0$ under $E_y$} & B_y \mbox{ happened,}\\
   \mbox{ arbitrary } & B_y \mbox{ did not happen}.
  \end{cases}
  $$
  
  Then $$\Pr_{Y \in [N], J \in [D]}[ E( f(Y), Y, J) \in S_0 ] < \frac{\lambda}{2} (1-\gamma) + \gamma < \lambda - \epsilon.$$
But $|S_0| = \lambda M$, and thus we get a contradiction to the $\epsilon$-extracting merger property of $E$. This completes the proof.
 \end{proof}

 \subsection{The structure of abnormal conductors}\label{sec42}

 The previous lemma gave us a $y$ for which $E_y$ is abnormal. 
 We now use show that abnormal conductors are very structured, and thus get a lower bound on $D$.

 \begin{lemma}
 \label{lem:abnormalstructure}
  Let $C: [N] \times [D] \to [M]$ be a $(\gamma, \lambda)$-abnormal conductor.
  Suppose $10 \epsilon < \lambda < \frac{1}{2}$.

  Suppose that for $X \in [N]$ and $J \in [D]$ picked uniformly and independently, $C(X, J)$ is $\epsilon$-close to the uniform distribution on $[M]$.
  Then
  $$D \geq \min\left\{\Omega\left(\frac{1}{\lambda} \log(\lambda \gamma M) \right), \Omega\left(\lambda \gamma M)^{1/4}\right)\right\}.$$
  \end{lemma}
\begin{proof}
 We begin with a pruning phase to remove the high degree vertices from the right side. At first reading, it will be helpful to consider the case where $B = \emptyset$.
 
 Let $\beta = \frac{\lambda}{5} - \epsilon$.
 Note that the average right degree is $ND/M$. 
 Define the set of high-degree right vertices by:
 $$ B = \{ z \in [M] \mid \mbox{ there are at least } \frac{1}{\beta} \frac{ND}{M} \mbox{ edges to $z$} \}.$$ 
 
 Thus $|B| \leq \beta M$.
 By the hypothesis on $C(X, J)$, we have
 $$ \Pr_{X \in [N], J \in [D]}[C(X,J) \in B] \leq \beta + \epsilon.$$
 
 Let $G$ be the set of all vertices on the left that do not have too many edges to $B$; namely:
 $$ G = \{ x \in [N] \mid x \mbox{ has at most $2(\beta + \epsilon) D$ edges to $B$ }\}.$$
 Then $|G| \geq N/2$.
 
 Now pick $S \in { [M] \choose \lambda M }$ uniformly at random.
 If there is a vertex $x \in G$ that totally misses $S \setminus B$, then by choice of $G$:
 $$ | \{ j \mid C(x,j) \in S \} | \leq 2(\beta + \epsilon) D < \frac{1}{2}\lambda D,$$
 namely, $x$ mostly misses $S$.
 
 By our hypothesis on the abnormality of $C$, the existence of such an $x$ cannot happen too often.
 Thus:
 \begin{align}
  \label{eqtotalmiss}
 \Pr_S [ \exists x \in G \mid x \mbox{ totally misses $S\setminus B$}] < 1-\gamma.
 \end{align}

 For each $x \in G$, let $A_x$ be the event that $x$ totally misses $S\setminus B$ under $C$.
 
 We are interested in the event that some $x \in G$ totally misses $S\setminus B$, namely, the event $\bigvee_{x \in G} A_x$.
  
 Observe that\footnote{Here we use the observation that $(1-\frac{D}{(1-\lambda)M}) < e^{-\frac{2D}{(1-\lambda)M}} < e^{-4D/M}$, which follows from the fact that $1-x < e^{-2x}$ for $x < 1/2$.} $$\Pr[A_x] \geq \frac{ {M - D \choose \lambda M}}{{M \choose \lambda M}}  \geq e^{-4\lambda D} =: p.$$
 Define $A = \sum_{x \in G} A_x$. 
 Then $\E[A] \geq |G| p.$
 
  By the second moment method, we have:
  $$ \Pr [A = 0] \leq \frac{\Var[A]}{\E[A]^2}.$$
  
  But Equation~\eqref{eqtotalmiss} tells us that $\Pr[A= 0] > \gamma$.
  
  Thus $\Var[A] \geq \gamma \E[A]^2 \geq \gamma \cdot  p^2 |G|^2$.
    
  We now extract some structure from this.
  
  We have:
  $$\Var[A] = \sum_{x, x' \in G} \left(\Pr[A_x \wedge A_{x'}] - \Pr[A_x] \Pr[A_{x'}] \right).$$
  
  Two simple observations about this expression:
  \begin{itemize}
   \item Each term in the sum above is at most $1$.  
\item  Furthermore, if $x, x'$ have no common neighbors in $[M]\setminus B$, then the corresponding term of the sum above is $\leq 0$. Indeed, if $U_x, U_{x'} \subseteq [M]\setminus B$ are the neighborhoods of $x$ and $x'$ in $[M]\setminus B$, and if they are disjoint, then:
  $$ \Pr[A_x] = \frac{  {M - |U_x| \choose \lambda M} }{{ M \choose \lambda M}},$$
  $$ \Pr[A_{x'}] = \frac{  {M - |U_{x'}| \choose \lambda M} }{{ M \choose \lambda M}},$$
  $$ \Pr[A_x \wedge A_{x'}] = \frac{  {M - |U_x \cup U_{x'}| \choose \lambda M} }{{ M \choose \lambda M}} = \frac{  {M - |U_x| - |U_{x'}| \choose \lambda M} }{{ M \choose \lambda M}}.$$
  So $$\frac{\Pr[A_x \wedge A_{x'}]}{\Pr[A_x]\Pr[A_{x'}]} = \prod_{i=0}^{\lambda M-1} \frac{(M-i) \cdot (M-|U_x| - |U_{x'}| - i)}{(M-|U_x|-i)\cdot(M-|U_{x'}|-i)} \leq 1.$$
  \end{itemize}
  
  Combining the largeness of $\Var[A]$ with these two observations tells us that there are many $x,x' \in G$ which have a common neighbor in $[M]\setminus B$.
  Specifically:
  $$  \gamma p^2 |G|^2\leq \Var[A] \leq \sum_{x,x' \in G} {\mathbf 1}[x, x' \mbox{ have a common neighbor in $[M]\setminus B$}].$$

  Thus there are at least $\gamma p^2 |G|^2 \geq \frac{1}{4} \gamma p^2 N^2$ pairs $x,x'$ from $G$ that have 
  a common neighbor in $[M] \setminus B$.

  Now the initial pruning we did will help us. Since all the vertices in $[M] \setminus B$ have degree at most $\frac{1}{\beta} \frac{ND}{M}$, we can bound the number of such pairs $x, x'$. For every vertex $x \in G$, there are at most $D \cdot \frac{1}{\beta} \frac{ND}{M}$  vertices $x'$ such that $x$ and $x'$ share a common neighbor in $[M] \setminus B$. Thus the total number of pairs $x,x'$ from $G$ that have a common neighbor in $[M] \setminus B$ is at most
  $$ N \cdot D \cdot \frac{1}{\beta}\frac{ND}{M} = \frac{1}{\beta} \frac{D^2}{M} N^2.$$
  
  Thus $\frac{1}{4} \gamma p^2 \leq \frac{1}{\beta} \frac{D^2}{M}$.
  Since $p = e^{-4\lambda D}$, we get:
  $$ M \leq \frac{4}{\gamma \beta} D^2 e^{8\lambda D}.$$
  This means that either $D \geq \Omega\left( \left(\gamma \beta M\right)^{1/4} \right) = \Omega\left( (\gamma \lambda M)^{1/4} \right)$,
  or else:
  $$D \geq \Omega\left( \frac{1}{\lambda}\log(\gamma \beta M ) \right) \geq \Omega\left( \frac{1}{\lambda}\log(\gamma \lambda M ) \right) .$$

\end{proof}

 \subsection{Putting everything together}\label{sec43}
 
 We now prove \Cref{thm:mergerlowerbound}.
 \begin{proof} 
 Let $E:[N]^2 \times [D] \to [M]$ be an $\epsilon$-extracting merger.
 
 Set $\lambda = 20 \epsilon$ and $\gamma = \epsilon$.
  \Cref{lem:findabnormal} tells us that there is some $y:=y_0$
 for which $E_y$ is $(\lambda, \gamma)$-abnormal.
 
 Now, since $E$ is $\epsilon$-extracting, we have that $E_y(X, J) = E(X,y,J)$ is $\epsilon$-close to the uniform distribution on $[M]$ for uniform and independent $X \in [N]$ and $J \in [D]$.
 Thus \Cref{lem:abnormalstructure} tells us that
 $$D \geq \min\left\{ \Omega\left(\frac{1}{\epsilon} \log(\epsilon^2 M) \right), \Omega\left( ( \epsilon^2 M)^{1/4} \right) \right\} $$.
 
 If $\epsilon \geq \frac{1}{M^{1/10}}$, then the first expression is smaller and 
 $$ D \geq \Omega(\frac{1}{\epsilon} \log M),$$ 
 and if $\epsilon < \frac{1}{M^{1/10}}$, then $E$ is also a $\frac{1}{M^{1/10}}$-extracting merger, and thus using the above lower bound for $\frac{1}{M^{1/10}}$ in place of $\epsilon$, we get that: 
 $$D \geq M^{\Omega(1)}.$$

\end{proof}

\newpage
\section{Seedless Extracting Multimergers and Projections of Partitions}\label{sec5}

In this section, we study seedless multimergers. Here our understanding
is far from complete, and we suggest many directions for research.

We begin by observing a connection between seedless multimergers and a
very natural and clean geometric question: 
how do we partition the unit cube $[0,1]^t$ into $c$ parts to ensure
that all $s$-dimensional
axis-parallel projections of all parts are small? We then prove some
interesting positive and negative results about special cases of
this general question. We conclude by collecting a number
of observations and questions about this natural partitioning problem.

\subsection{Seedless Multimergers with 1 bit output}\label{sec52}

In \Cref{sec31} we have already seen that there are no seedless mergers (i.e., with $s = 1$).
We now look into seedless multimergers.

Let us consider the simplest nontrivial situation: $t = 3$ and $s = 2$, and $m=1$
(we only try to extract $1$ bit of randomness), with $n$ big.
Suppose a given function $E: (\zo{n})^3 \to \{0,1\}$ is known to be a
$(\epsilon,s)$-multimerger. 
For convenience, we identify $\zo{n}$ with $[N]$, for $N = 2^n$.

By the multimerger property, for every function $f: [N]^2 \to [N]$,
the distribution of $E(X, Y, f(X,Y))$ should be $\epsilon$-close to uniform.
Let
$$P_{XY,0} = \{(x,y) \in [N]^2 : \exists z \mid E(x,y,z) = 0\}.$$
Notice that this is the projection of $E^{-1}(0)$ to two coordinates.

If $P_{XY,0}$ is bigger than $\frac{1+\epsilon}{2} N^2$, then we can violate the  multimerger property:
we define $f: [N]^2 \to [N]$ by $f(x,y) =$ $z$, if any, such that $E(x,y,z) = 0$".
Then $E(X,Y, f(X,Y))$ for uniform and independent $X,Y \in [N]$ is $\epsilon$-far from uniform.

We have a similar observation for all the other two dimensional projections, and also for the set $E^{-1}(1)$.
Thus if a seedless one-bit output multimerger for $3$-part $2$-where random sources exists,
then there is a partition of $[N]^3$ into 2 parts such that each part has all its 2-dimensional axis parallel projections
have size at most $\frac{1+\epsilon}{2} N^2$.

The connection also goes in reverse.
Suppose we have a partition $A,B$ of $[N]^3$ for which each part has all its $2$-dimensional axis parallel projections with size at most $\frac{1 + \epsilon}{2} N^2$. Let $E:[N]^3 \to \{0,1\}$ be the unique function with $E^{-1}(0) = A$ and $E^{-1}(1) = B$. Suppose $(X,Y,Z)$ is an $[N]^3$-valued random variable that is $2$-where random. Then we claim that $E(X,Y,Z)$ is $\epsilon$-close to the uniform distribution. Indeed, if $(X,Y)$ is uniformly distributed over $[N]^2$ (the cases of
$(Y,Z)$ and $(X,Z)$ being uniformly distributed are similar), then:
\begin{align*}
\Pr[E(X,Y,Z) = 0] &\leq \Pr_{X, Y} [ \exists z \in [N] \mbox{ s.t. } E(X,Y,z) = 0 ]\\
&\leq \Pr_{X, Y} [ \exists z \in [N] \mbox{ s.t. } (X,Y,z) \in A ]\\
&\leq \frac{|\Pi_{XY}(A)|}{N^2}\\
&\leq \frac{1 + \epsilon}{2},\\
\Pr[E(X,Y,Z) = 1] &\leq \Pr_{X, Y} [ \exists z \in [N] \mbox{ s.t. } E(X,Y,z) = 1 ]\\
&\leq \Pr_{X, Y} [ \exists z \in [N] \mbox{ s.t. } (X,Y,z) \in B ]\\
&\leq \frac{|\Pi_{XY}(B)|}{N^2}\\
&\leq \frac{1 + \epsilon}{2},\\
\end{align*}
which implies the desired $\epsilon$-closeness to uniform of $E(X,Y,Z)$.

The exact same argument applies to general $t, s$. We record this below.

\begin{theorem}
Let $N = 2^n$.
There exists a seedless $(n,d,t,m,\epsilon,s)$-multimerger if and only if there
is a partition of $[N]^t$ into two sets $A,B$ such that for every subset $U \subseteq [t]$
of size $s$, the projections $\Pi_U(A)$ and $\Pi_U(B)$ onto the coordinates $U$ are of size
at most $\frac{1 + \epsilon}{2}N^s$.
\end{theorem}

Motivated by this, we consider general projections of partitions questions, where the set $[N]^t$ is partitioned
into $c$ parts, and we seek to minimize the maximum $s$-dimensional axis parallel projection of all the parts\footnote{For $c > 2$, the problem of getting such partitions with $c$ parts is somewhat related to the problem of multimergers with $\log_2(c)$ bit output, but the connection is not as tight as for the case of $c = 2$}.

As noted in the introduction, there is a basic bound for this problem that comes from Shearer's lemma. It says
that there is a lower bound of $\left(\frac{1}{c} \right)^{s/t} N^s$ on the size of some projection. This 
bound is usually not tight -- but it sometimes is! Whenever $c$ is a perfect $t$'th power, then this bound is tight,
and is realized by a partition into product sets. But for other $c$ this kind of partition does not work, and very interesting questions
ensue. In particular, we would like to highlight the case of $N$ gigantic, and $c = \poly(t)$ (so that $c$ is clearly not a perfect $t$'th power). 

Here is one observation that gives a flavor of what happens for large $t$ and $s$. When $c$ is a constant, and $s = t - o(\sqrt{t})$, there is a partition
so that all $s$-dimensional projections of size $\frac{1}{c} + o(1)$. This comes by considering suitable threshold partitions.
Extensions of this are related to the BKKKL~\cite{BKKKL,Friedgut} conjectures on low influence functions.

This question is also equivalent to a problem in the continuous domain about open covers of $[0,1]^t$.
Here we want to minimize the maximum $s$-dimensional projection size when we cover
$[0,1]^t$ by $c$ open sets. 

In the following sections, we discuss two results on partitioning in three dimensions. 
For the first result, we get (to our surprise!) the tight bound for partitioning the cube into two parts.
For the second result, we get nontrivial bounds (both upper and lower) for partitioning the cube into three parts.

\paragraph{Apology:} These questions are more naturally phrased as questions about covers rather than partitions. However we stick to the partition language because of the particular sequence of events that led us to these problems.

\section{Partitioning the 3-dimensional cube into two parts}\label{sec54}

In this section, we prove a tight bound on the largest 2 dimensional axis parallel projection 
of a part when partitioning $[0,1]^3$ into $2$ parts.

Let $\pi_{XY}, \pi_{YZ}, \pi_{XZ}: [0,1]^3 \to [0,1]^2$ be the 2-dimensional projection maps.

The following example gives a nice partitioning with small projections.

\begin{definition}\label{defn-maj-ext-3k} (Majority Partitioning Scheme)\\

We define the function $MAJ_{3}: [0,1]^{3} \rightarrow \zo{} $ as  
\begin{center}$MAJ_{3} (x,y,z)= Maj(x_1, y_1, z_1)$\end{center}
where $Maj$ denotes the Majority function on 3 bits, and where $x_1, y_1, z_1$ denote the indicator variables for whether $x > 1/2$, $y>1/2$, $z>1/2$ respectively.
 
We refer to the partition naturally induced by the output of $MAJ_{3}$ on the input space $ [0,1]^{3}$ i.e. $\{MAJ^{-1}_{3}(0), MAJ^{-1}_{3}(1)\}$, as the Majority Partitioning Scheme.
\end{definition}

We next record the observation tha all $2$-dimensional axis-parallel projections of all parts in the majority partitioning scheme on  $ [N]^{3}$ are of size at most $\frac{3}{4}N^2$, which is stated in the following lemma:

\begin{lemma}\label{lem-maj-proj}(Majority Partitions Optimally) 

Every $2$-dimensional projection of every partition in the majority partitioning scheme $MAJ_{3}$ on $ [0,1]^{3}$  is of size at most $\frac{3}{4}$.
\end{lemma}

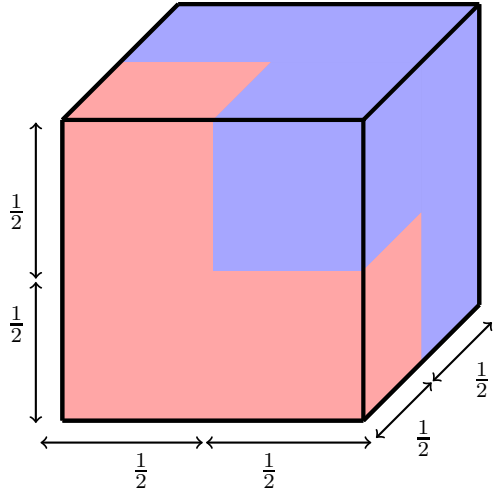
\begin{figure}
\centering 
\begin{tikzpicture}[scale=0.4]
  \foreach \x in{0,...,10}
{   
\draw[ultra thick](0,\x ,10) -- (10,\x ,10);
\draw[ultra thick](\x ,0,10) --(\x ,10,10);
\draw[ultra thick](10,\x ,10) -- (10,\x ,0);
\draw[ultra thick](\x ,10,10) -- (\x ,10,0);
\draw[ultra thick](10,0,\x ) -- (10,10,\x );
\draw[ultra thick](0,10,\x ) -- (10,10,\x );
} 
\fill[blue!35,ultra thick] (10,10,0) -- (0,10,0)--(0,10,5)--(10,10,5);
\fill[blue!35,ultra thick] (10,10,0) -- (10,10,5)--(10,0,5)--(10,0,0);

\fill[blue!35,ultra thick] (5,10,10) -- (10,10,10)--(10,10,5)--(5,10,5);

\fill[red!35,ultra thick] (0,10,10) -- (5,10,10)--(5,10,5)--(0,10,5);

\fill[red!35,ultra thick] (0,5,10) -- (0,10,10)--(5,10,10)--(5,5,10);
\fill[blue!35,ultra thick] (10,5,10) -- (10,5,5)--(10,10,5)--(10,10,10);

\fill[red!35,ultra thick] (0,5,10) -- (10,5,10)--(10,0,10)--(0,0,10);
\fill[red!35,ultra thick] (10,0,5) -- (10,5,5)--(10,5,10)--(10,0,10);

\fill[blue!35,ultra thick] (5,5,10) -- (10,5,10)--(10,10,10)--(5,10,10);

\fill[blue!35,ultra thick] (10,10,0) -- (0,10,0)--(0,10,5)--(10,10,5);
\fill[blue!35,ultra thick] (10,10,0) -- (10,10,5)--(10,0,5)--(10,0,0);

\fill[blue!35,ultra thick] (5,10,10) -- (10,10,10)--(10,10,5)--(5,10,5);

\fill[red!35,ultra thick] (0,10,10) -- (5,10,10)--(5,10,5)--(0,10,5);

\fill[red!35,ultra thick] (0,5,10) -- (0,10,10)--(5,10,10)--(5,5,10);
\fill[blue!35,ultra thick] (10,5,10) -- (10,5,5)--(10,10,5)--(10,10,10);

\fill[red!35,ultra thick] (0,5,10) -- (10,5,10)--(10,0,10)--(0,0,10);
\fill[red!35,ultra thick] (10,0,5) -- (10,5,5)--(10,5,10)--(10,0,10);

\fill[blue!35,ultra thick] (5,5,10) -- (10,5,10)--(10,10,10)--(5,10,10);
\fill[blue!35,ultra thick] (10,10,0) -- (0,10,0)--(0,10,5)--(10,10,5);
\fill[blue!35,ultra thick] (10,10,0) -- (10,10,5)--(10,0,5)--(10,0,0);

\fill[blue!35,ultra thick] (5,10,10) -- (10,10,10)--(10,10,5)--(5,10,5);

\fill[red!35,ultra thick] (0,10,10) -- (5,10,10)--(5,10,5)--(0,10,5);

\fill[red!35,ultra thick] (0,5,10) -- (0,10,10)--(5,10,10)--(5,5,10);
\fill[blue!35,ultra thick] (10,5,10) -- (10,5,5)--(10,10,5)--(10,10,10);

\fill[red!35,ultra thick] (0,5,10) -- (10,5,10)--(10,0,10)--(0,0,10);
\fill[red!35,ultra thick] (10,0,5) -- (10,5,5)--(10,5,10)--(10,0,10);

\fill[blue!35,ultra thick] (5,5,10) -- (10,5,10)--(10,10,10)--(5,10,10);

\draw [thick, <->] (11, 0,1.5) -- (11,0,6.6) node[midway, below  right] {$\frac 12$};
\draw [thick, <->] (11, 0,6.7) -- (11,0,11.5) node[midway, below  right] {$\frac 12$};

\draw [thick, <->] (12, 1,14.5) -- (6.5,1,14.5) node[midway, below  left] {$\frac 12$};
\draw [thick, <->] (1, 1,14.5) -- (6.4,1,14.5) node[midway, below  right] {$\frac 12$};

\draw [thick, <->] (-0.5,10.3,11) -- (-0.5,5.1,11) node[near end, above  left] {$\frac 12$};
\draw [thick, <->] (-0.5, 0.3,11) -- (-0.5,5,11) node[midway, above  left] {$\frac 12$};

\draw[ultra thick](10, 0,0) -- (10, 0,10);
\draw[ultra thick](10, 0,10) -- (0, 0,10);
\draw[ultra thick](0, 0,10) -- (0, 10,10);
\draw[ultra thick](0, 10,10) -- (0, 10,0);
\draw[ultra thick](0, 10,0) -- (10, 10,0);
\draw[ultra thick](10, 0,0) -- (10, 10,0);
\draw[ultra thick](10, 10,0) -- (10, 10,10)--(10,0,10);
\draw[ultra thick](10, 10,10) -- (0, 10,10);

\end{tikzpicture}
\caption{Majority Partitioning of the cube into $2$ parts, where the partitioned sets are coloured in red and blue. Observe that $all$ projections of the red set and the blue set are of equal size $\frac 34$, and \Cref{lem-3-2-2parts} implies this partitioning is optimal.}
\end{figure}

In the other direction, we first prove a lower bound on projection sizes 
for a discrete version of the problem.

Let $N$ be a large positive integer.
We reuse notation and let $\pi_{XY}, \pi_{YZ}, \pi_{XZ}: [N]^3 \to [N]^2$ be the 2-dimensional projection maps.

\begin{theorem}\label{lem-3-2-2parts}
Let $A, B \subseteq [N]^3$ be a partition.

Then one of the six 2-dimensional projections of $A$ and $B$
$$ \pi_{XY}(A), \pi_{YZ}(A), \pi_{XZ}(A),
\pi_{XY}(B), \pi_{YZ}(B), \pi_{XZ}(B)$$
has size at least $\frac{3}{4} N^2.$
\end{theorem}
\begin{proof}
Suppose $\pi_{XY}(A)$ and $\pi_{XY}(B)$ are both at most $\frac{3}{4} N^2$.

Fix $z \in [N]$. Let us consider the slice $S_z = [N]^2 \times \{z\}$,
and focus on the $X$ and $Y$ projections of the sets $A \cap S_z$ and $B \cap S_z$ (so four projections in all, each being a subset of $[N]$). 

Define\footnote{If $A,B$ was merely a cover of $[N]^3$ rather than a partition, the correct definition would be $$A_{Xz} = \{ x \mid \not\exists y \in [N] \mbox{ s.t. } (x,y,z) \in B\},$$ etc, and the rest of the proof would remain the same.}:
$$A_{Xz} = \{x \mid   \forall y \in [N], (x,y,z) \in A \}.$$
$$B_{Xz} = \{x \mid   \forall y \in [N], (x,y,z) \in B \}.$$
$$A_{Yz} = \{y \mid   \forall x \in [N], (x,y,z) \in A \}.$$
$$B_{Yz} = \{y \mid   \forall x \in [N], (x,y,z) \in B \}.$$

Let $\alpha_{Xz}, \beta_{Xz}, \alpha_{Yz}, \beta_{Yz} \in [0,1]$ be
their fractional sizes (= size divided by N).

Then we have the following:
\begin{itemize}
\item  $A_{Xz} \cap B_{Xz} = \emptyset$ and $A_{Yz} \cap B_{Yz} = \emptyset$.
Thus:
\begin{align}
\alpha_{Xz} + \beta_{Xz} \leq 1,\\
\alpha_{Yz} + \beta_{Yz} \leq 1.
\end{align}

\item $$\left( (A_{Xz} \times [N]) \cup ([N] \times A_{Yz})\right) \subseteq \pi_{XY}(A).$$
This is because any $(x,y) \in (A_{Xz} \times [N])$ has $(x,y,z) \in A$, and thus 
$(x,y) \in \pi_{XY}(A)$.

The fractional size of the left hand side is $1-(1- \alpha_{Xz})(1-\alpha_{Yz})$, and the fractional
size of the right hand side is $\leq 3/4$.

This gives us (after applying the AM-GM inequality\footnote{For any non-negative real numbers $x$ and $y$, $\sqrt{x\cdot y} \leq \frac{x+y}2 $}):
\begin{align}\label{eq-eta}
\alpha_{Xz} + \alpha_{Yz} \leq 1.
\end{align}

\item Similarly, $$\left( (B_{Xz} \times [N]) \cup ([N] \times B_{Yz})\right) \subseteq \pi_{XY}(B),$$

\begin{align}
\beta_{Xz} + \beta_{Yz} \leq 1.
\end{align}

\item At most one of $A_{Xz}, B_{Yz}$ can be nonempty, and at most one of 
$A_{Yz}, B_{X,z}$ can be nonempty. This is because $x \in A_{Xz}$ and
$y \in B_{Yz}$ imply that $(x,y,z) \in A$ and $(x,y,z) \in B$ respectively.
Thus at most one of  $\alpha_{Xz},\beta_{Yz}$, and at most one of 
$\alpha_{Yz}, \beta_{Xz}$ can be nonzero. 

\end{itemize}

Putting everything together, we get that only two of the four numbers
$\alpha_{Xz}, \beta_{Xz}, \alpha_{Yz}, \beta_{Yz}$ can be nonzero, and furthermore,
the sum of those two is bounded above by $1$.

Therefore, for each $z \in [N]$,
$$\alpha_{Xz} + \beta_{Xz} + \alpha_{Yz} + \beta_{Yz} \leq 1.$$

Averaging in $z$, we get that
$$\E_z [ \alpha_{Xz} + \beta_{Xz} + \alpha_{Yz} + \beta_{Yz}] \leq 1,$$
and thus one of the four numbers:
$$\E_z[\alpha_{Xz}], \E_z[\beta_{Xz}], \E_z[\alpha_{Yz}], \E_z[\beta_{Yz}]$$
is at most $1/4$.

Finally, observe that $(1-\E_z[\alpha_{Xz}])$ is the fractional size of $\pi_{XZ}(B)$ 
(and similarly for the other three numbers), and so one of the four projections
$$\pi_{XZ}(B), \pi_{XZ}(A), \pi_{YZ}(B), \pi_{YZ}(A)$$
has size at least $\frac{3}{4} N^2$.
\end{proof}

By a simple discretization argument, we get the following corollary:
\begin{corollary}
Any cover of $[0,1]^3$ by two open sets $A, B$ has one of the following 6 sets:
$$ \Pi_{XY}(A), \Pi_{YZ}(A), \Pi_{XZ}(A),  \Pi_{XY}(B), \Pi_{YZ}(B), \Pi_{XZ}(B)$$
having area at least $3/4$.
\end{corollary}

Thus we  get that $MAJ_{3}$  is an {\it optimal} partition for partitioning $[N]^3$ into two parts.

\section{Partitioning the 3-dimensional cube into three parts} \label{sec6}

In this section we study the case of partitioning the 3 dimensional cube $[0,1]^3$ into $3$ parts.

We begin with a nice partition of $[0,1]^3$ into 3 parts so that each part has small $2$-dimensional axis-parallel projections.

\begin{definition}\label{defn-maj-GR} (Golden Ratio Partitioning Scheme)\\
Let $u$ be the positive root of $x^2+x=1$. We define the function $GR_{3}: [0,1]^{3} \rightarrow \{0,1,2\} $ as  
$$
GR_3(x,y,z)=
\begin{cases}
0,& |x|>u,|y|>u\\
1,& |x|\leq u, |y|\leq u, |z|\leq  \frac 12\\
2,  &\text{otherwise.}
\end{cases}
$$
 
We refer to the partition into 3 parts naturally induced by the output of $GR_{3}$ on the input space $ [0,1]^{3}$ i.e. $\{GR^{-1}_{3}(0), GR^{-1}_{3}(1), GR^{-1}_{3}(2)\}$, as the golden ratio partitioning scheme.
\end{definition}

Observe that all $2$-dimensional projections of all partitions in the golden ratio partitioning scheme on  $ [0,1]^{3}$ are of size $u\leq 0.619$, which is stated in the following lemma:

\begin{lemma}\label{lem-gold-proj}(Golden Ratio Partitioning Bound) 
Every $2$-dimensional projection of every partition in the golden ratio partitioning scheme $GR_{3}$ on $ [0,1]^{3}$  is of size $u \leq 0.619$.
\end{lemma}

\begin{figure}
\centering 
\begin{tikzpicture}[scale=0.4]
  \foreach \x in{0,...,10}
{   

\draw[](0,\x ,10) -- (10,\x ,10);
\draw[](\x ,0,10) --(\x ,10,10);
\draw[](10,\x ,10) -- (10,\x ,0);
\draw[](\x ,10,10) -- (\x ,10,0);
\draw[](10,0,\x ) -- (10,10,\x );
\draw[](0,10,\x ) -- (10,10,\x );
} 
\fill[blue!35,ultra thick] (10,10,0) -- (4,10,0)--(4,10,6)--(10,10,6);
\fill[blue!35,ultra thick] (10,10,0) -- (10,10,6)--(10,0,6)--(10,0,0);

\fill[red!35,ultra thick] (4,10,10) -- (10,10,10)--(10,10,6)--(4,10,6);
\fill[red!35,ultra thick] (0,10,0) -- (4,10,0)--(4,10,10)--(0,10,10);

\fill[red!35,ultra thick] (4,10,10) -- (10,10,10)--(10,10,6)--(4,10,6);
\fill[red!35,ultra thick] (0,10,0) -- (4,10,0)--(4,10,10)--(0,10,10);

\fill[red!35,ultra thick] (4,10,10) -- (10,10,10)--(10,10,6)--(4,10,6);
\fill[red!35,ultra thick] (0,10,0) -- (4,10,0)--(4,10,10)--(0,10,10);

\fill[red!35,ultra thick] (0,5,10) -- (10,5,10)--(10,10,10)--(0,10,10);
\fill[red!35,ultra thick] (10,5,10) -- (10,5,6)--(10,10,6)--(10,10,10);
\fill[green!35,ultra thick] (0,5,10) -- (10,5,10)--(10,0,10)--(0,0,10);
\fill[green!35,ultra thick] (10,0,10) -- (10,0,6)--(10,5,6)--(10,5,10);
\draw [thick, <->] (11, 0,1.5) -- (11,0,8) node[midway, below  right] {u};
\draw [thick, <->] (11, 0,8.1) -- (11,0,11) node[midway, below  right] {1-u};

\draw [thick, <->] (11, 11,1.5) -- (4.7,11,1.5) node[midway, above  left] {u};
\draw [thick, <->] (0.7, 11,1.5) -- (4.6,11,1.5) node[midway, above ] {1-u};

\draw [thick, <->] (-0.5, 10.3,11) -- (-0.5, 5.3,11) node[near end, above  left] {$\frac 12$};
\draw [thick, <->] (-0.5, 0.3,11) -- (-0.5,5.2,11) node[midway, above  left] {$\frac 12$};

\draw[ultra thick](10, 0,0) -- (10, 0,10);
\draw[ultra thick](10, 0,10) -- (0, 0,10);
\draw[ultra thick](0, 0,10) -- (0, 10,10);
\draw[ultra thick](0, 10,10) -- (0, 10,0);
\draw[ultra thick](0, 10,0) -- (10, 10,0);
\draw[ultra thick](10, 0,0) -- (10, 10,0);
\draw[ultra thick](10, 10,0) -- (10, 10,10)--(10,0,10);
\draw[ultra thick](10, 10,10) -- (0, 10,10);
\end{tikzpicture}
\caption{Golden Ratio Partitioning of the cube into $3$ parts, where the partitioned sets are coloured in red, green and blue. The green and red parts are just translates of each other. Here $u$ is the positive root of $x^2+x=1.$}
\end{figure}
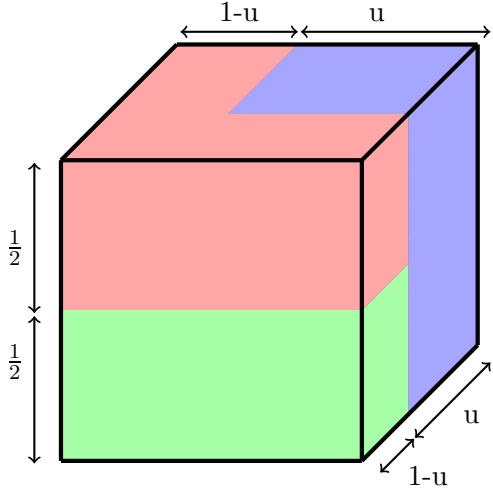

We do not know if this is the optimal partition into 3 parts.

For the rest of this section, we prove the best lower bound that we know.
As in the previous section, we do this via an analogous discrete problem.

Let $\etao \approx 0.5264$ be the real number $\in [0.5, 1.0]$
satisfying:
$$
 (2-3\etao) \cdot \left(2 - 2\sqrt{1- \etao}\right) + (3\etao - 1)  = \frac{1}{6} (4 - \etao).$$

\begin{theorem}\label{thm-3-2-3parts}
Let $A,B,C \subseteq [N]^{3} $  be a $3$-partition of  $[N]^{3}$. Then one of the nine $2$-dimensional projections of $A$, $B$ and $C$, i.e., 

$\Pi_{XY}(A),\Pi_{XY}(B),\Pi_{XY}(C), \Pi_{YZ}(A),\Pi_{YZ}(B), \Pi_{YZ}(C), \Pi_{XZ}(A),\Pi_{XZ}(B), \Pi_{XZ}(C)$ has size at least $\etao N^2$.

\end{theorem}

Before embarking on the proof of \Cref{thm-3-2-3parts}, we note down a simple set intersection lemma that will be useful.
\begin{lemma}[Set intersection inequality]
Suppose $U,V,W$ be arbitrary sets which have union equal to $T$.

Then
$$|U| + |V| + |W|  \geq 2|T| - \left(|U \setminus (V \cup W) | + |V \setminus (W \cup U) |+|W \setminus (U \cup V) | \right) +  |U \cap V \cap W|.$$
\end{lemma}

This lemma gives a way to get a lower bound on the average size of three sets $U,V,W$ that cover a set $T$
by first proving an upper bound on the sizes of the ``unique" parts $U\setminus (V \cup W), V \setminus (U \cup W), W \setminus (U \cup V)$.
The proof is simple and omitted.

We now prove \Cref{thm-3-2-3parts}.
\begin{proof}

Consider any partition $A$, $B$ and $C$  of $[N]^3$ into 3 parts. 
Suppose $\Pi_{XY}(A)$, $\Pi_{XY}(B)$ and $\Pi_{XY}(C)$ are at most $ \etao$. 
(If not we are done).

Fix $z \in [N]$.

Our first step is to consider the slice $S_z = [N]^2 \times \{z\}$, and
focus on the $X$ and $Y$ projections of the 3 sets $A \cap S_z, B \cap S_z, C \cap S_z$ (so six projections in all, each being a subset of $[N]$).

Define:
$$A_{Xz} = \{x \in [N] \mid \exists y \in [N] \mbox{ s.t. } (x,y,z) \in A \}.$$
$$B_{Xz} = \{x \in [N] \mid \exists y \in [N] \mbox{ s.t. } (x,y,z) \in B \}.$$
$$C_{Xz} = \{x \in [N] \mid \exists y \in [N] \mbox{ s.t. } (x,y,z) \in C \}.$$
$$A_{Yz} = \{y \in [N] \mid \exists x \in [N] \mbox{ s.t. } (x,y,z) \in A \}.$$
$$B_{Yz} = \{y \in [N] \mid \exists x \in [N] \mbox{ s.t. } (x,y,z) \in B \}.$$
$$C_{Yz} = \{y \in [N] \mid \exists x \in [N] \mbox{ s.t. } (x,y,z) \in C \}.$$

Note that:
$$ A_{Xz} \cup B_{Xz} \cup C_{Xz} = [N] $$
$$ A_{Yz} \cup B_{Yz} \cup C_{Yz} = [N] $$
since $A, B, C$ is a partition of $[N]^3$.

Next we indentify the ``pure" parts of these projections, defined below:

$$ \PA_{Xz} = A_{Xz} \setminus (B_{Xz} \cup C_{Xz} )$$
$$ \PB_{Xz} = B_{Xz} \setminus (C_{Xz} \cup A_{Xz} )$$
$$ \PC_{Xz} = C_{Xz} \setminus (A_{Xz} \cup B_{Xz} )$$
$$ \PA_{Yz} = A_{Yz} \setminus (B_{Yz} \cup C_{Yz} )$$
$$ \PB_{Yz} = B_{Yz} \setminus (C_{Yz} \cup A_{Yz} )$$
$$ \PC_{Yz} = C_{Yz} \setminus (A_{Yz} \cup B_{Yz} )$$

Furthermore, we have:
$$ \{x\mid \Pi^{-1}_{XZ} (x,z) \subseteq A\} \subseteq \PA_{Xz}$$
and five similar containments for $\PB_{Xz}, \PC_{Xz}, \PA_{Yz}, \PB_{Yz}, \PC_{Yz}$.

Let $\Palpha_{Xz},\Pbeta_{Xz},\Pgamma_{Xz},\Palpha_{Yz},\Pbeta_{Yz},\Pgamma_{Yz} \in [0,1]$ be their fractional sizes.

Note that since the corresponding sets are disjoint, we have:
\begin{align}
\label{Psumbound1}
 \Palpha_{Xz} + \Pbeta_{Xz} + \Pgamma_{Xz} \leq 1\\
\label{Psumbound2}
 \Palpha_{Yz} + \Pbeta_{Yz} + \Pgamma_{Zz} \leq 1
\end{align}

\begin{lemma}
\label{3partsupport}
 For any $z \in [N]$, out of the $6$ variables $\Palpha_{Xz}, \Pbeta_{Xz}, \Pgamma_{Xz}, \Palpha_{Yz}, \Pbeta_{Yz}, \Pgamma_{Yz}$, let $H$ be the set of those variables that are nonzero.
Then $H$ is a subset of at least one of the following sets of variables:
$$ 
\{\Palpha_{Xz}, \Palpha_{Yz}\},
\{\Pbeta_{Xz}, \Pbeta_{Yz}\},
\{\Pgamma_{Xz}, \Pgamma_{Yz}\},
\{\Palpha_{Xz}, \Pbeta_{Xz}, \Pgamma_{Xz}\},
\{\Palpha_{Yz}, \Pbeta_{Yz}, \Pgamma_{Yz}\} $$
\end{lemma}
\begin{proof}
 It is a consequence of the easy observation that $\Palpha_{Xz}$ and $\Pbeta_{Yz}$ cannot both be nonzero (and 5 similar easy observations).
\end{proof}

Let
$$\delta_{Xz} =  \begin{cases}  1 & \Palpha_{Yz}, \Pbeta_{Yz}, \Pgamma_{Yz} > 0 \\ 0 & \mbox{otherwise} \end{cases}.$$
$$\delta_{Yz} =  \begin{cases}  1 & \Palpha_{Xz}, \Pbeta_{Xz}, \Pgamma_{Xz} > 0 \\ 0 & \mbox{otherwise} \end{cases}.$$
Note that $\delta_{Xz}$ depends on the projections in the $Y$ direction (and vice versa).
The reason for this definition is the following observation: if $\delta_{Xz} = 1$, then we have
\begin{align}
A_{Xz} \cap B_{Xz} \cap C_{Xz} = A_{Xz} = B_{Xz} = C_{Xz} = [N],
\end{align}
and similarly, if $\delta_{Yz} = 1$, then we have
\begin{align}
A_{Yz} \cap B_{Yz} \cap C_{Yz} = A_{Yz} = B_{Yz} = C_{Yz} = [N],
\end{align}

which is something that our set intersection lemma can exploit.

Define 
$$\lambda_{Xz} = \Palpha_{Xz} + \Pbeta_{Xz} + \Pgamma_{Xz} - \delta_{Xz},$$
$$\lambda_{Yz} = \Palpha_{Yz} + \Pbeta_{Yz} + \Pgamma_{Yz} - \delta_{Yz}.$$
$$\lambda_{z} = \lambda_{Xz} + \lambda_{Yz}.$$

Note that by Equations~\eqref{Psumbound1}, \eqref{Psumbound2}, for all $z$,
\begin{align}
\label{lambdabound1}
\lambda_{Xz} \leq 1.\\
\label{lambdabound2}
\lambda_{Yz} \leq 1.
 \end{align}

By the set intersection lemma,
\begin{align*}
|A_{Xz}| + |B_{Xz}| + |C_{Xz}|  &\geq 2N - \left(|\PA_{Xz}| + |\PB_{Xz}| + |\PC_{Xz}|\right) + |A_{Xz} \cap B_{Xz} \cap C_{Xz}| \\
& \geq (2 - \lambda_{Xz}) N
\end{align*}
Similarly,
\begin{align*}
|A_{Yz}| + |B_{Yz}| + |C_{Yz}|  \geq (2 - \lambda_{Yz}) N
\end{align*}

Summing over $z \in [N]$ and adding these two equations, we get:
\begin{align}
\label{eqlambdaX}
\Pi_{XZ}(A) + \Pi_{XZ}(B) + \Pi_{XZ}(C) &\geq \left(2 - \E_{z}[\lambda_{Xz}] \right) N^2,\\
\label{eqlambdaY}
\Pi_{YZ}(A) + \Pi_{YZ}(B) + \Pi_{YZ}(C)  &\geq \left(2 - \E_{z}[\lambda_{Yz}] \right) N^2,\\
\label{eqlambda}
\Pi_{XZ}(A) + \Pi_{XZ}(B) + \Pi_{XZ}(C) + 
\Pi_{YZ}(A) + \Pi_{YZ}(B) + \Pi_{YZ}(C)  &\geq \left(4 - \E_{z}[\lambda_z] \right) N^2.
\end{align}

Our goal is now to get an upper bound on  $\E_{z}[\lambda_z]$.

To get our main result, we will show that $\E_{z}[\lambda_z] \leq \lambda^* := 4 - 6 \etao \approx 0.856$ (or else we find a large projection in some other way). This will show that one of the 6 projections on the left hand side is at least $\etao N^2$, as desired.

If we just want to get a projection of size $\geq \frac{1}{2} N^2$, then it suffices to show that $\lambda^* \leq 1$, and this turns out to be simpler.

Towards that end, we define $\alpha_X$ to be the fraction of $x$ for which $\{x\} \times [N] \subseteq \Pi_{XY}(A)$.
Similarly define $\alpha_Y, \beta_X, \beta_Y, \gamma_X, \gamma_Y$.

Note that since $\PA_{Xz} \times [N] \times \{z\} \subseteq A$, we have:
$$\alpha_{Xz} \leq \alpha_X,$$
and 5 similar inequalities.

Note that $\alpha_X \leq \eta$, and 5 similar inequalities.

Define $u: [0,1] \to [0,2]$ by:
$$ u(a) = 2 - 2 \sqrt{1-a}.$$
Using the argument used to arrive at \Cref{eq-eta} (by the AM-GM inequality), we have 
$$\Palpha_{Xz} + \Palpha_{Yz} \leq u(\etao) \quad\quad \mbox{ ( and thus } \alpha_X + \alpha_Y \leq u(\etao) \mbox{ ).}$$
and 2 similar pairs of inequalities.

Let
$$g_X = \max \{ \alpha_X + \beta_X, \beta_X + \gamma_X, \alpha_X + \gamma_X\}.$$
$$g_Y = \max \{ \alpha_Y + \beta_Y, \beta_Y + \gamma_Y, \alpha_Y + \gamma_Y\}.$$

By the inequalities above, we have:
$$ g_X + g_Y \leq 2 \etao + u(\etao).$$

Now let 
$$q_X = \Pr_{z \in [n]} [\mbox{exactly two of } \alpha_{Xz},\beta_{Xz},\gamma_{Xz} \mbox{ are nonzero}].$$
$$q_Y = \Pr_{z \in [n]} [\mbox{exactly two of } \alpha_{Yz},\beta_{Yz},\gamma_{Yz} \mbox{ are nonzero}].$$
$$q = \Pr_{z \in [n]} [\mbox{at most one of } \alpha_{Xz},\beta_{Xz},\gamma_{Xz} \mbox{ and at most one of } \alpha_{Yz},\beta_{Yz},\gamma_{Yz} \mbox{ is nonzero}]$$

We are now in a position to state a key lemma which will prove our lower bound:

\begin{lemma}$$\E_{z}[ \lambda_z] \leq  q \cdot u(\etao) + q_X \cdot \min(g_X,1) + q_Y\cdot \min(g_Y,1)).$$
\end{lemma}
\begin{proof}
Let $z \in [N]$. We take cases on which of the 6 numbers $\Palpha_{Xz}, \Pbeta_{Xz}, \Pgamma_{Xz}, \Palpha_{Yz}, \Pbeta_{Yz}, \Pgamma_{Yz}$ are nonzero.

By Lemma~\ref{3partsupport}, there only a few cases to consider.
\begin{itemize}
\item If the first three numbers are nonzero or the second three numbers are nonzero, then $\lambda_z$ is nonpositive because the sum of those three is at most $1$ (by Equations~\eqref{Psumbound1}, \eqref{Psumbound2}, and $\delta_{Xz} = 1$ or $\delta_{Yz} = 1$.

\item If exactly two of the first three numbers are nonzero, then
$\lambda_z$ is at most $\min(g_X,1)$.

This happens for $q_X$ fraction of the $z$'s.

\item If exactly two of the second three numbers are nonzero, then
$\lambda_z$ is at most $\min(g_Y,1)$.

This happens for $q_Y$ fraction of the $z$'s
\item If at most one of the first three numbers and at most one of the second three numbers is nonzero, then $\lambda_z$ is at most $2-2\sqrt{1-\etao}$.

This happens for $q$ fraction of the $z$'s.
\end{itemize} 
\end{proof}

Now $q + q_X + q_Y \leq 1$.
At this point, we already see that $\E_{z}[\lambda_z] \leq 1$ (since $u(\etao) \approx  0.6237 \leq 1$), and this gives us the result that some projection has size at least $\frac{1}{2}N^2$.

To get our improved bound of $\etao N^2$, we need one more idea.
\begin{lemma}
 \begin{align}
 q_X + \E_{z}[\lambda_{Yz}] \leq 1\\
 q_Y + \E_{z}[\lambda_{Xz}] \leq 1
 \end{align}
\end{lemma}
\begin{proof}
We prove the first inequality, the second being similar.

$q_X$ is the fraction of $z$ for which exactly two of $\{\Palpha_{Xz}, \Pbeta_{Xz}, \Pgamma_{Xz} \}$ are nonzero. For such a $z$, we have $\Palpha_{Yz}=\Pbeta_{Yz}=\Pgamma_{Yz} = \delta_{Yz} = 0$, and thus $\lambda_{Yz} = 0$.

Along with Equations~\eqref{lambdabound1}, \eqref{lambdabound2}, this completes the proof of the lemma.
\end{proof}

By Equation~\eqref{eqlambdaX}, if $\E_{z}[\lambda_{Xz}]$ is at most $2 - 3 \etao$, then we get a projection onto the $XZ$ plane of size at least $\etao N^2$, and we are done. 
Similarly, by Equation~\eqref{eqlambdaY}, if $\E_z[\lambda_{Yz}]$ is at most $2-3\etao$, then we get a projection onto the $YZ$ plane of size at least $\etao N^2$, and we are done.
Thus we may assume that both 
$\E_{z}[\lambda_{Xz}]$ and $\E_z[\lambda_{Yz}]$ are at least $2-3\etao$.

By the previous lemma, we thus get that $q_X, q_Y \leq 3 \etao - 1$.

Summarizing everything we know:
$$ g_X + g_Y \leq 2 \etao + u(\etao)$$
$$ q_X, q_Y \leq 3 \etao - 1.$$
$$ q + q_X + q_Y \leq 1.$$
Under these constraints, we claim that:
$$ q \cdot u(\etao) + q_X \cdot \min(g_X,1) + q_Y \cdot \min(g_Y, 1) \leq \lambda^*.$$
By inspection, we see that the LHS is maximized when:
$$ g_X = 1, q_X = 3\etao - 1, q = 1-q_X = 2-3\etao,$$
which makes it evaluate to:
$$ (2-3\etao) \cdot u(\etao) + (3\etao - 1)  = \frac{1}{6} (4 - \etao) = \lambda^*,$$
where the first equality is the defining equation of $\etao$, and the second equality is the definition of $\lambda^*$. This completes the proof.

\end{proof}

This gives us a corresponding result about covers of $[0,1]^3$ with 3 open sets.

\begin{corollary}
Any cover of $[0,1]^3$ by three open sets $A, B, C$ has one of the following 9 sets:
$$ \Pi_{XY}(A), \Pi_{YZ}(A), \Pi_{XZ}(A), 
 \Pi_{XY}(B), \Pi_{YZ}(B), \Pi_{XZ}(B),
 \Pi_{XY}(C), \Pi_{YZ}(C), \Pi_{XZ}(C)$$
having area at least $\etao$.
\end{corollary}

\section{Acknowledgements}
Part of this work was done when Vishvajeet was a visitor at the Institute for Advanced Study, Princeton and thanks Avi Wigderson for the same.

Swastik thanks K.P.S. Bhaskara Rao for immeasurable help with measurable sets and more.

We thank the anonymous reviewers for helpful comments on the writeup.

\nocite{*}
\bibliographystyle{alpha}
\bibliography{general}

\clearpage

\end{document}